\documentclass[12pt]{amsart}
\usepackage{amssymb,bbold}
\usepackage[all]{xy}

\theoremstyle{plain}
\newtheorem{theorem}{Theorem}
\newtheorem{corollary}[theorem]{Corollary}

\newtheorem{proposition}[theorem]{Proposition}

\theoremstyle{definition}
\newtheorem{definition}[theorem]{Definition}
\newtheorem{remark}[theorem]{Remark}
\newtheorem{example}[theorem]{Example}

\renewcommand{\le}{\leqslant}
\renewcommand{\ge}{\geqslant}

\begin{document}
\title[Topological algebras of locally solid vector subspaces]
{Topological algebras of locally solid vector subspaces of order bounded operators}
\author{Omid Zabeti }
\address{Omid Zabeti,
Department of Mathematics, Faculty of Mathematics, University of Sistan and Baluchestan,
P.O.Box 98135-674, Zahedan, Iran.}
\email{\text{o.zabeti@gmail.com}}
\begin{abstract}
Let $E$ be a locally solid vector lattice. In this paper, we consider two particular vector subspaces of the space of all order bounded operators on $E$. With the aid of two appropriate topologies, we show that under some conditions, they establish both, locally solid vector lattices and topologically complete topological algebras.
\end{abstract}
\maketitle
\section{introductory facts}
First of all, let us recall some notation and terminology used in this paper.
A subset $S$ of a vector lattice $E$ is said to be \textit{solid} if $y\in S$, $x\in E$, and $|x|\leq|y|$, then, we have $x\in S$. A \textit{topological vector lattice} is an ordered topological vector space which is also a vector lattice; for more details on these notions see \cite{NA,ROB}. By a \textit{locally solid vector lattice} we mean a topological vector lattice with a locally solid topology. Also, note that a vector lattice $E$ is \textit{Dedekind complete} if every subset of $E$ which is bounded above has a supremum. By $[-a,a]$ in vector lattice $E$, we mean the order interval consists of all $z\in E$ that $-a\le z\le a$. Finally, consider this point that in a topological vector lattice $E$, we have two notions for boundedness; a subset $B\subseteq E$ is called order bounded if it is contained in an order interval and it is bounded if for each zero neighborhood $U\subseteq E$ there exists a positive scalar $\gamma$ with $B\subseteq \gamma U$. Some relations and results for the mentioned types of bounded sets and operators defined on $E$ have been investigated; see \cite{HO, HO1, HO2, KH} for more details. In addition, for terminology and standard facts concerning topological vector lattices and related concepts, we refer the reader to \cite{AB, NB, ZA}. 

Now, let us state some motivation.
Let $E$ be a vector lattice and $B_b(E)$ be the space of all order bounded operators on $E$. Since we have just order structure on $B_b(E)$, it is natural that we could not expect a topological structure on it. Nevertheless, when $E$ is a Banach lattice, $B_b(E)$ forms a Banach lattice, too (see \cite{AW}). But when we discuss topological vector lattices or even locally solid vector lattices, no "suitable" topology is known for $B_b(E)$. Here, by "suitable" topology, we mean a topology which respects continuity in algebraic or lattice structures or converts $B_b(E)$ to a known topological algebraic structure. On the other hand,
the relation between algebraic and topological structures has been considerable for several decades in the sense that it has many applications in other disciplines (see \cite{AA, AB, AB1, HO, HO1}). In fact, combining these notions is fundamental for modern analysis, for example Banach spaces, Banach lattices and operators between them.

 Although, there is no appropriate topology for order bounded operators on topological vector lattices in the sense that it makes the space to a known topological algebraic structure; instead, we can investigate two subsets of the class of all order bounded operators which possess suitable topologies. In fact, in this paper, we consider two subspaces of the space of all order bounded operators on a topological vector lattice; using both topological and order structures of the underlying set. With the two effective topologies, we show that these classes of order bounded operators on a Dedekind complete topologically complete locally solid vector lattice, form both, locally solid vector lattices and also topologically complete topological algebras.
 
All vector spaces in this paper are assumed to be real and all topological vector lattices are considered to be locally solid.
In this step, let us start with the definition.
\begin{definition}
Let $E$ and $F$ be locally solid vector lattices. A linear operator $T:E \to F$ is said to be:
\begin{itemize}
\item[\em i.]{$no$-bounded if there exists some zero neighborhood $U\subseteq E$ such that  $T(U)$ is order bounded in $F$}.
\item[\em ii.]{$bo$-bounded if for every bounded set $ B\subseteq E$, $T(B)$ is order bounded in $F$}.
\end{itemize}
\end{definition}
The class of all $no$-bounded operators on a locally solid vector lattice $E$ is denoted by $B_{no}(E)$ and is equipped with the topology of order uniform convergence on some zero neighborhood. Namely, a net $(S_{\alpha})$ of $no$-bounded operators converges to zero in this topology if there exists a zero neighborhood $U\subseteq E$ such that for any $a\in E_{+}$ there is an $\alpha_0$ with $S_{\alpha}(U)\subseteq [-a,a]$ for each $\alpha\ge\alpha_0$.
The class of all $bo$-bounded operators on a locally solid vector lattice $E$ is denoted by $B_{bo}(E)$ and is allocated to the topology of order uniform convergence on bounded sets; a net $(S_{\alpha})$ of $bo$-bounded operators order converges to zero uniformly on a bounded set $B\subseteq E$ if for each $a\in E_{+}$ there is an $\alpha_0$ with $S_{\alpha}(B)\subseteq [-a,a]$ for each $\alpha\ge\alpha_0$. One can easily verify that $B_{no}(E)$ and $B_{bo}(E)$ are subalgebras of the algebra of all order bounded operators on a topological vector lattice $E$. Note that these algebras are not unital in general; when $E$ contains an order bounded zero neighborhood, then these algebras are unital. Nevertheless, this condition is sufficient; consider Theorem 2.2 from \cite{HO} and results after that.
In addition, by Theorem 2.19 in \cite{AB}, every order bounded subset is bounded so that every $no$-bounded operator is $bo$-bounded and every $bo$-bounded operator is an order bounded operator. In prior to anything, we show that these classes of linear operators are not equal, in general.
\begin{example}\label{1}
Let $E$ be $c_0$, the space of all null sequences, with the usual order and norm topology. Consider the identity operator $I$ on $E$. Indeed, $I$ is order bounded. But it fails to be $bo$-bounded. Suppose $N_{1}^{(0)}$ is the closed unit ball centered at zero with radius one. It is not difficult to see that the sequence $(u_n)$ defined via $u_n=\Sigma_{i=1}^{n}e_i$ is not order bounded in $E$, in which $e_i$ is the standard basis of $E$.
\end{example}
Also, Example 2.5 from \cite {HO}, presents a $bo$-bounded operator which is not $no$-bounded.
\begin{remark}
Note that one can consider $bo$-bounded operators on a locally solid vector lattice $E$ as a class of bounded operators from one bornological space ($E$ with the bornology of topologically bounded sets) to another bornological space ($E$ with the bornology of order bounded sets). For a review on terminology of bornological spaces and related notions, see \cite{HOG}. When one discusses with topological vector spaces, the bounded sets are usually considered as topologically bounded sets; which is called the Von Neumann bornology (see \cite{HOG}, 1:4.3). Therefore, taking the difference between topologically bounded sets and order bounded ones into account, we can not expect a suitable topology for the class of all bounded operators on bornological spaces in a manner that it preserves continuity in the theme of $bo$-bounded operators between locally solid vector lattices. On the other hand, there is a notion for convergence in any bornological vector space in term of bounded sets; namely, a sequence $(x_n)$ in a bornological vector space $E$ is convergent to zero if there exist a bounded set $B$ and a sequence $(\lambda_n)$ of reals tending to zero such that $x_n\in \lambda_n B$ for each $n\in \Bbb N$ (see \cite{HOG}, 1:4.1). It is known that this convergence in  topological vector spaces is almost topological convergence; this means that in the most cases in analysis, these two types of convergence coincide (\cite{HOG}, 1:4.3, Proposition 2). Also, it is worth mentioning that the order convergence topology on $B_{bo}(E)$ is finer ( and so suitable for our purpose) than the usual topology of uniform convergence on bounded sets ( see \cite{Tr} for more information).
\end{remark}
\section{main result}
\begin{theorem}
The operations of addition, scalar multiplication, and product are continuous in $B_{no}(E)$ with respect to the order uniform convergence topology on some zero neighborhood.
\end{theorem}
\begin{proof}
Suppose $(T_{\alpha})$ and $(S_{\alpha})$ are two nets of $no$-bounded operators which are order uniform convergent to zero on some zero neighborhood $U\subseteq E$.
For any $a\in E_{+}$ there are some $\alpha_0$ and $\alpha_1$ with $T_{\alpha}(U)\subseteq [\frac{-a}{2},\frac{a}{2}]$ for each $\alpha\ge\alpha_0$ and $S_{\alpha}(U)\subseteq [\frac{-a}{2},\frac{a}{2}]$ for each $\alpha\ge\alpha_1$. There exists an $\alpha_2$ such that $\alpha_2\ge\alpha_0$ and $\alpha_2\ge\alpha_1$, so that for each $\alpha\ge\alpha_2$,
\[(T_{\alpha}+S_{\alpha})(U)\subseteq T_{\alpha}(U)+S_{\alpha}(U)\subseteq [\frac{-a}{2},\frac{a}{2}]+[\frac{-a}{2},\frac{a}{2}]=[-a,a].\]
Now, we show the continuity of the scalar multiplication. Let $(\gamma_n)$ be a sequence of reals which is convergent to zero. For sufficiently large $n$, we have, $|\gamma_n|\le 1$, so that $\gamma_nU\subseteq U$. Thus, using the net $(T_{\alpha})$ as in the first part, we have $\gamma_n T_{\alpha}(U)\subseteq T_{\alpha}(U)\subseteq [-a,a]$.
Now, we show that product is also continuous. Since every order bounded set is bounded, there is a $\gamma>0$ with $[\frac{-a}{\gamma},\frac{a}{\gamma}]\subseteq U$. There exists an $\alpha_3$ with $S_{\alpha}(U)\subseteq [\frac{-a}{\gamma},\frac{a}{\gamma}]$ for each $\alpha\geq \alpha_3$. Choose $\alpha_4$ such that $T_{\alpha}(U)\subseteq [-a,a]$ for each $\alpha\geq\alpha_4$. Pick index $\alpha_5$ with $\alpha_5\geq\alpha_3$ and $\alpha_5\geq\alpha_4$. Therefore, for each $\alpha\geq\alpha_5$,  we have
\[T_{\alpha}(S_{\alpha}(U))\subseteq T_{\alpha}([\frac{-a}{\gamma},\frac{a}{\gamma}])\subseteq T_{\alpha}(U)\subseteq [-a,a].\]
\end{proof}
\begin{theorem}
The operations of addition, scalar multiplication, and product are continuous in $B_{bo}(E)$ with respect to the order uniform convergence topology on bounded sets.
\end{theorem}
\begin{proof}
Suppose $(T_{\alpha})$ and $(S_{\alpha})$ are two nets of $bo$-bounded operators which are order uniform convergent to zero on bounded sets.
Fix a bounded set $B\subseteq E$. For any positive $a\in E$ there are some $\alpha_0$ and $\alpha_1$ with $T_{\alpha}(B)\subseteq [\frac{-a}{2},\frac{a}{2}]$ for each $\alpha\ge\alpha_0$ and $S_{\alpha}(B)\subseteq [\frac{-a}{2},\frac{a}{2}]$ for each $\alpha\ge\alpha_1$. There exists an $\alpha_2$ such that $\alpha_2\ge\alpha_0$ and $\alpha_2\ge\alpha_1$, so that for each $\alpha\ge\alpha_2$,
\[(T_{\alpha}+S_{\alpha})(B)\subseteq T_{\alpha}(B)+S_{\alpha}(B)\subseteq [\frac{-a}{2},\frac{a}{2}]+[\frac{-a}{2},\frac{a}{2}]=[-a,a].\]
Now, we show the continuity of the scalar multiplication. Let $(\gamma_n)$ be a sequence of reals which is convergent to zero. Consider $F=\cup_{n\in \Bbb N} \gamma_n B$ which is indeed bounded in $E$. There exists an $\alpha_3$ such that $T_{\alpha}(F)\subseteq [-a,a]$ for each $\alpha\geq\alpha_3$. Thus, we have  $\gamma_nT_{\alpha}(B)
\subseteq T_{\alpha}(F)\subseteq [-a,a]$.
Now, we show that product is also continuous. There exists some $\alpha_4$ with $S_{\alpha}(B)\subseteq [-a,a]$ for each $\alpha\geq\alpha_4$. Since every order bounded set in a locally solid vector lattice is bounded, there is an $\alpha_5$ such that $T_{\alpha}([-a,a])\subseteq [-a,a]$ for each $\alpha\geq\alpha_5$. Choose index $\alpha_6$ that $\alpha_6\geq\alpha_4$ and $\alpha_6\geq\alpha_5$. For each $\alpha\geq\alpha_6$, we conclude
\[T_{\alpha}(S_{\alpha}(B))\subseteq T_{\alpha}([-a,a])\subseteq [-a,a].\]
\end{proof}
In this step, we investigate the continuity of the lattice operations on these classes of order bounded operators with respect to the assumed topologies.
\begin{theorem}
Suppose $E$ is Dedekind complete. Then, the lattice operations in $B_{no}(E)$ are continuous with respect to the order uniform convergence topology on some zero neighborhood.
\end{theorem}
\begin{proof}
 By the Riesz-Kantorovich formula, for $no$-bounded operators $T,S\in B_{no}(E)$ and every $x\in E_{+}$, we have
 \[(T\vee S)(x)=\sup\{T(u)+S(v): u,v\geq 0, u+v=x\}.\]
  Now, suppose $(T_{\alpha})$ and $(S_{\alpha})$ are two nets of $no$-bounded operators that order converge uniformly on some zero neighborhood $U\subseteq E$ to the linear operators $T$ and $S$, respectively. Fix $x\in U_{+}$, and suppose $u,v$ are positive elements such that $x=u+v$. Since $U$ is solid, we have $u,v\in U$. Also, recall that for two subsets $A,B$ in a vector lattice, we have $\sup(A)-\sup(B)\le\sup(A-B)$. Thus,
 \[\sup\{T_{\alpha}(u)+S_{\alpha}(v): u,v\geq 0, u+v=x\}-\sup\{T(u)+S(v): u,v\geq 0, u+v=x\}\]
 \[\le\sup\{(T_{\alpha}-T)(u)+(S_{\alpha}-S)(v):u,v\geq 0, u+v=x\}.\]
 Let $a\in E_{+}$ be arbitrary. There exists an $\alpha_0$ such that $(T_{\alpha}-T)(U)\subseteq [-\frac{a}{2},\frac{a}{2}]$ for each $\alpha\geq\alpha_0$. Pick index $\alpha_1$ with $(S_{\alpha}-S)(U)\subseteq[-\frac{a}{2},\frac{a}{2}]$ for each $\alpha\geq\alpha_1$. Choose $\alpha_2$ such that $\alpha_2\geq\alpha_0$ and $\alpha_2\geq\alpha_1$. So, for each $\alpha\geq\alpha_2$,
\[(T_{\alpha}\vee S_{\alpha})(x)-(T\vee S)(x)\leq(T_{\alpha}-T)(x)+(S_{\alpha}-S)(x)\in [-a,a].\]

\end{proof}
\begin{theorem}
Suppose $E$ is Dedekind complete. Then, the lattice operations in $B_{bo}(E)$ are continuous with respect to the  order uniform convergence topology on bounded sets.
\end{theorem}
\begin{proof}
 By the Riesz-Kantorovich formula, for $bo$-bounded operators $T,S\in B_{bo}(E)$ and every $x\in E_{+}$, we have
 \[(T\vee S)(x)=\sup\{T(u)+S(v): u,v\geq 0, u+v=x\}.\]
  Now, suppose $(T_{\alpha})$ and $(S_{\alpha})$ are two nets of $bo$-bounded operators that order converge uniformly on bounded sets to the linear operators $T$ and $S$, respectively. Fix a bounded set $B\subseteq E$. Put,
   \[B_1=\{u\in E^{+}, \exists v\in E^{+}, u+v=x, \text{ for some } x\in B^{+}\}.\]
   One can verify that $B_1$ is also bounded, so that $T_{\alpha}\to T$ and $S_{\alpha}\to S$ order converge uniformly on $B_1$.
   Fix a $x\in B_{+}$ and some positive element $u$ and $v$ that $x=u+v$. Thus,
 \[\sup\{T_{\alpha}(u)+S_{\alpha}(v): u,v\geq 0, u+v=x\}-\sup\{T(u)+S(v): u,v\geq 0, u+v=x\}\]
 \[\le\sup\{(T_{\alpha}-T)(u)+(S_{\alpha}-S)(v):u,v\geq 0, u+v=x\}.\]
 Choose $a\in E_{+}$ arbitrary. There exists an $\alpha_0$ such that $(T_{\alpha}-T)(B_1)\subseteq [-\frac{a}{2},\frac{a}{2}]$ for each $\alpha\geq\alpha_0$. Pick index $\alpha_1$ with $(S_{\alpha}-S)(B_1)\subseteq[-\frac{a}{2},\frac{a}{2}]$ for each $\alpha\geq\alpha_1$. Choose $\alpha_2$ such that $\alpha_2\geq\alpha_0$ and $\alpha_2\geq\alpha_1$. Thus, for each $\alpha\geq\alpha_2$,
\[(T_{\alpha}\vee S_{\alpha})(x)-(T\vee S)(x)\leq(T_{\alpha}-T)(x)+(S_{\alpha}-S)(x)\in [-a,a].\]
\end{proof}
Now, we show that $B_{no}(E)$ and $B_{bo}(E)$ are topologically complete algebras with respect to the assigned topologies.
\begin{proposition}\label{4}
Suppose $(T_{\alpha})$ is a net of $no$-bounded operators which is order convergent uniformly on some zero neighborhood to the linear operator $T$. Then, $T$ is also $no$-bounded.
\end{proposition}
\begin{proof}
There is a zero neighborhood $U\subseteq E$ such that for any $a\in E_{+}$ there exists an $\alpha_0$ with $(T-T_{\alpha})(U)\subseteq [-a,a]$ for each $\alpha\ge\alpha_0$. There is a zero neighborhood $U_1\subseteq U$ such that $T_{\alpha_{0}}(U_1)$ is order bounded. Since $T(U_1)\subseteq T_{\alpha_{0}}(U_1)+[-a,a]$, we conclude that $T$ is also $no$-bounded.
\end{proof}
\begin{proposition}\label{5}
Suppose $(T_{\alpha})$ is a net of $bo$-bounded operators which is order convergent uniformly on bounded sets to the linear operator $T$. Then, $T$ is also $bo$-bounded.
\end{proposition}
\begin{proof}
Fix a bounded set $B\subseteq E$. Take a positive $a\in E$. There exists an $\alpha_0$ with $(T-T_{\alpha})(B)\subseteq [-a,a]$ for each $\alpha\ge\alpha_0$. Note that  $T_{\alpha_{0}}(B)$ is order bounded. Therefore from the relation $T(B)\subseteq T_{\alpha_{0}}(B)+[-a,a]$, we conclude that $T$ is also $bo$-bounded.
\end{proof}
\begin{proposition}
Let $E$ be a topologically complete vector lattice. Then, $B_{no}(E)$ is complete with respect to the topology of order uniform convergence on some zero neighborhood.
\end{proposition}
\begin{proof}
Suppose $E$ is complete and $(T_{\alpha})$ is a Cauchy net in $B_{no}(E)$. Let $W$ be an arbitrary zero neighborhood in $E$. For any $a\in E_{+}$ there is a positive scalar $\gamma$ such that $[\frac{-a}{\gamma},\frac{a}{\gamma}]\subseteq W$. There exist an $\alpha_0$ and some zero neighborhood $U\subseteq E$ with $(T_{\alpha}-T_{\beta})(U)\subseteq [\frac{-a}{\gamma},\frac{a}{\gamma}]$ for each $\alpha\ge\alpha_0$ and for each $\beta\ge\alpha_0$. For any $x\in E$, there is a positive real $\eta$ such that $x\in \eta U$ so that we conclude that $(T_{\alpha}(x))$ is a Cauchy net in $E$. Put $T(x)=\lim T_{\alpha}(x)$.
Since this convergence happens in $B_{no}(E)$, Proposition \ref{4} leads to the desired result.
\end{proof}
\begin{proposition}
Let $E$ be a topologically complete vector lattice. Then, $B_{bo}(E)$ is complete with respect to the topology of order uniform convergence on bounded sets.
\end{proposition}
\begin{proof}
Suppose $E$ is complete and $(T_{\alpha})$ is a Cauchy net in $B_{bo}(E)$. Let $W$ be an arbitrary zero neighborhood in $E$. For any positive $a\in E$ there is a positive scalar $\gamma$ such that $[\frac{-a}{\gamma},\frac{a}{\gamma}]\subseteq W$. Fix a  bounded set $B\subseteq E$. There exists an index $\alpha_0$ with $(T_{\alpha}-T_{\beta})(B)\subseteq [\frac{-a}{\gamma},\frac{a}{\gamma}]$ for each $\alpha\ge\alpha_0$ and for each $\beta\ge\alpha_0$. Since singletons are bounded, for any $x\in E$ and sufficiently large $\alpha$ and $\beta$, we conclude that $(T_{\alpha}(x))$ is a Cauchy net in $E$. Put $T(x)=\lim T_{\alpha}(x)$. By Proposition \ref{5}, $T$ is also $bo$-bounded.
\end{proof}
Using the above results and Theorem 2.17 in \cite{AB}, we have the following.
\begin{corollary}
Suppose $E$ is a Dedekind complete topologically complete locally solid vector lattice. Then, both $B_{no}(E)$ and $B_{bo}(E)$, establish locally solid vector lattices and complete topological algebras at the same time.
\end{corollary}
 \bibliographystyle{amsplain}
 
\end{document}